\documentclass[12pt, leqno, twoside]{article}
\usepackage{amsmath, amsthm, amsfonts, amssymb, graphicx}
\usepackage{enumerate}

\numberwithin{equation}{section}
     \newtheorem{thm}{Theorem}[section]
     \newtheorem{cor}[thm]{Corollary}
     \newtheorem{prop}[thm]{Proposition}
     \newtheorem{lem}[thm]{Lemma}
\theoremstyle{definition}
      \newtheorem{defn}{Definition}[section]
     \newtheorem{exmp}{Example}[section]
\theoremstyle{remark}
     \newtheorem{rem}{Remark}[section]

\newcommand{\C}{\mathbb{C}}

\newcommand{\cF}{\mathcal{F}}

\newcommand{\cL}{\mathcal{L}}

\newcommand{\cX}{\mathcal{X}}

\newcommand{\BMO}{\mathrm{BMO}}

\newcommand{\Lip}{\mathrm{Lip}}

\newcommand{\ls}{\lesssim}
\newcommand{\gs}{\gtrsim}

\newcommand{\sqc}[1]{(#1_n)_{n\ge0}}    

\newcommand{\Lz}[1]{L_{#1,0}}
\newcommand{\cLN}{\mathcal{L}^{\natural}}

\newcommand{\pwm}{\mathrm{PWM}}
\newcommand{\BMON}{\mathrm{BMO}^{\natural}}
\newcommand{\LipN}{\mathrm{Lip}^{\natural}}

\newcommand{\msckw}{%
\footnotetext{\hspace{-0.35cm} 2010 {\it Mathematics Subject Classification}. 
Primary 60G46; Secondary 46E30, 42B35.
\endgraf{\it Key words and phrases.} 
martingale, pointwise multiplier, Campanato space, bounded mean oscillation.
}
}

\begin{document}

\baselineskip=18pt

\title{Pointwise multipliers on martingale Campanato spaces \msckw}
\author{Eiichi Nakai and Gaku Sadasue}
\date{}

\maketitle

\begin{abstract}
We introduce generalized Campanato spaces $\mathcal{L}_{p,\phi}$ 
on a probability space $(\Omega,\mathcal{F},P)$,
where $p\in[1,\infty)$ and $\phi:(0,1]\to(0,\infty)$.
If $p=1$ and $\phi\equiv1$, then $\mathcal{L}_{p,\phi}=\mathrm{BMO}$.
We give a characterization of the set of all pointwise multipliers 
on $\mathcal{L}_{p,\phi}$.
\end{abstract}

\section{Introduction}\label{sec:intro}

We consider a probability space $(\Omega,\cF,P)$ such that $\cF=\sigma(\bigcup_{n}\cF_n)$,
where $\{\cF_n\}_{n\ge0}$ is a nondecreasing sequence of sub-$\sigma$-algebras of $\cF$. 
For the sake of simplicity, let $\cF_{-1}=\cF_0$.
We suppose that every $\sigma$-algebra $\cF_n$ is generated by countable atoms,
where $B\in\cF_n$ is called an atom (more precisely a $(\cF_n,P)$-atom), 
if any $A\subset B$ with $A\in\cF_n$ satisfies $P(A)=P(B)$ or $P(A)=0$.
Denote by $A(\cF_n)$ the set of all atoms in $\cF_n$.
The expectation operator and the conditional expectation operators relative to $\cF_n$
are denoted by $E$ and $E_n$, respectively.

Let $\cX$ be a normed space of $\cF$-measurable functions.
We say that 
an $\cF$-measurable function $g$ is a pointwise multiplier on $\cX$,
if the pointwise multiplication
$fg$ is in $\cX$ for any $f\in\cX$.
We denote by $\pwm(\cX)$ the set of all pointwise multipliers on $\cX$.
If $\cX$ is a Banach space and has the following property,
then every $g\in\pwm(\cX)$ is a bounded operator on $\cX$. 
\begin{equation}\label{sub}
 \text{$f_n\to f$ in $\cX$ $(n\to\infty)$} \ \Longrightarrow\ 
 \text{$\exists\{n(j)\}$ s.t. $f_{n(j)}\to f$ a.s. $(j\to\infty)$.}
\end{equation}
Actually, from \eqref{sub} we see that $g$ is a closed operator.
Therefore, $g$ is a bounded operator by the closed graph theorem.

It is known that $\pwm(L_p)=L_{\infty}$ for $p\in(0,\infty]$.
More generally, if $\cX$ is a (quasi) Banach function space,
then $\pwm(\cX)=L_{\infty}$ (see \cite{Maligranda-Persson1989, Nakai1995MACT}).
For Banach function spaces, see Kikuchi~\cite{Kikuchi2010}.

In this paper we consider the pointwise multipliers on generalized Campanato spaces
which are not Banach function spaces in general.
We always assume that $\cF_0=\{\emptyset,\Omega\}$, 
that is, the operator $E_0$ coincides with $E$. 
Then we introduce generalized Campanato spaces 
$\cL_{p,\phi}$ and $\cLN_{p,\phi}$ 
as the following:
\begin{defn}\label{defn:Cam}
Let $p\in[1,\infty)$ and $\phi$ be a function from $(0,1]$ to $(0,\infty)$.
For $f\in L_1$, let
\begin{equation}\label{Cam-norm}
  \|f\|_{\cL_{p,\phi}}
  =
  \sup_{n\ge0}\sup_{B\in A(\cF_n)}
  \frac{1}{\phi(P(B))}\left(\frac{1}{P(B)}\int_B|f-E_nf|^p\,dP\right)^{1/p},
\end{equation}
and
\begin{equation}\label{CamN-norm}
  \|f\|_{\cLN_{p,\phi}}
  =
  \|f\|_{\cL_{p,\phi}}
  +
  |Ef|.
\end{equation}
Define
\begin{equation*}
  \cL_{p,\phi} = \{f\in L_1: \|f\|_{\cL_{p,\phi}}<\infty \}
 \quad\text{and}\quad
  \cLN_{p,\phi} = \{f\in L_1: \|f\|_{\cLN_{p,\phi}}<\infty \}.
\end{equation*}
\end{defn}

If $\phi(r)=r^{\lambda}$, $\lambda\in(-\infty,\infty)$, 
we simply denote $\cL_{p,\phi}$ and $\cLN_{p,\phi}$ 
by $\cL_{p,\lambda}$ and $\cLN_{p,\lambda}$, respectively,
which were introduced by \cite{Nakai-Sadasue2012JFSA}.

Note that $\cL_{p,\phi}$ and $\cLN_{p,\phi}$ coincide
as sets of measurable functions.
We regard $\cL_{p,\phi}=(\cL_{p,\phi},\|\cdot\|_{\cL_{p,\phi}})$ is a seminormed space
and $\cLN_{p,\phi}=(\cLN_{p,\phi},\|\cdot\|_{\cLN_{p,\phi}})$ is a normed space.
Then $\cLN_{p,\phi}$ is a Banach space, 
but it is not a Banach function space in general.
It is easy to see that $\cLN_{p,\phi}$ has the property \eqref{sub}, 
since 
\begin{equation*}
 \|f\|_{L_1}\le E[|f-Ef|]+|Ef|\le\max(1,\phi(1))\|f\|_{\cLN_{p,\phi}}.
\end{equation*}
For $g\in\pwm(\cLN_{p,\phi})$,
let
$$
 \|g\|_{Op}
 =
 \sup_{f\not\equiv0}
 \frac{\|fg\|_{\cLN_{p,\phi}}}{\|f\|_{\cLN_{p,\phi}}}.
$$

We also define $\BMO$ and $\Lip_{\alpha}$ as the following:
\begin{defn}\label{defn:BMO Lip}
For $\phi\equiv1$, 
denote $\cL_{1,\phi}$ and $\cLN_{1,\phi}$ by $\BMO$ and $\BMON$,
respectively.
For $\phi(r)=r^{\alpha}$, $\alpha>0$, 
denote $\cL_{1,\phi}$ and $\cLN_{1,\phi}$ by $\Lip_{\alpha}$ and $\LipN_{\alpha}$,
respectively.
\end{defn}

Let 
$$
 \Lz{1}=\{f\in L_1: Ef=0\}.
$$
Then $\BMO\cap\Lz{1}=\BMON\cap\Lz{1}$ and 
$\Lip_{\alpha}\cap\Lz{1}=\LipN_{\alpha}\cap\Lz{1}$. 
These spaces coincide with $\BMO$ and $\Lip_{\alpha}$ 
defined by Weisz~\cite{Weisz1990,Weisz1994}, respectively,
under the assumption that 
every $\sigma$-algebra $\cF_n$ is generated by countable atoms,
see \cite{Nakai-Sadasue2012JFSA} for details.

We say $\{\cF_n\}_{n\ge0}$ is regular if there exists $R\geq 2$ such that
\begin{equation}\label{regular}
  f_n\leq Rf_{n-1} \text{ for all non-negative martingales }f=\sqc{f}.
\end{equation}
A function $\theta:(0,1]\to(0,\infty)$ is said to  
satisfy the doubling condition
if there exists a constant $C>0$ such that 
\begin{equation*} 
     \frac1C\le\frac{\theta(r)}{\theta(s)}\le C 
     \quad\text{for}\quad r,s\in(0,1], \ \frac{1}{2}\le\frac{r}{s}\le 2.
\end{equation*} 
A function $\theta:(0,1]\to(0,\infty)$ is said to be 
almost increasing (almost decreasing)
if there exists a constant $C>0$ such that 
\begin{equation*} 
     \theta(r)\le C\theta(s) \quad (\theta(r)\ge C\theta(s)) 
     \quad\text{for}\quad 0<r\le s\le1.
\end{equation*}

Our main result is the following:
\begin{thm}\label{thm:PWM}
Let $\{\cF_n\}_{n\ge0}$ be regular, $\cF_0=\{\emptyset,\Omega\}$,
$p\in[1,\infty)$ and $\phi:(0,1]\to(0,\infty)$.
Assume that $\phi$ satisfies the doubling condition
and that 
\begin{equation}\label{int phi}
 \int_0^r\phi(t)^p\,dt\le Cr\phi(r)^p
 \quad\text{for all $r\in(0,1]$}.
\end{equation}
Let
\begin{equation}\label{phi_*}
 \phi_*(r)=1+\int_r^1\frac{\phi(t)}t\,dt.
\end{equation}
Then
\begin{equation*}
  \pwm(\cLN_{p,\phi})=\cL_{p,\phi/\phi_*}\cap L_{\infty}.
\end{equation*}
Moreover, for $g\in\pwm(\cLN_{p,\phi})$, $\|g\|_{Op}$ is equivalent to 
$\|g\|_{\cL_{p,\phi/\phi_*}}+\|g\|_{L_{\infty}}$.
\end{thm}

See \cite{Janson1976,Nakai1993Studia,Nakai-Yabuta1985JMSJ,Stegenga1976,Yabuta1993} for 
pointwise multipliers on $\BMO$ and Campanato spaces defined on the Euclidean space.
Our basic idea comes from \cite{Janson1976,Nakai-Yabuta1985JMSJ}.

\begin{rem}\label{rem:phi}
\begin{enumerate}[\upshape (i)]
\item 
If $\phi$ satisfies the doubling condition and \eqref{int phi},
then $r\phi(r)^p$ is almost increasing.
\item
If $\phi$ is almost increasing, then $\phi/\phi_*$ is also.
\item
Let
\begin{equation}\label{norm F}
 \|f\|_{\cL_{p,\phi,\cF}}
 =
 \sup_{n\ge0}\sup_{A\in \cF_n}
  \frac{1}{\phi(P(A))}\left(\frac{1}{P(A)}\int_A|f-E_nf|^p\,dP\right)^{1/p}.
\end{equation}
Then $\|f\|_{\cL_{p,\phi}}\le\|f\|_{\cL_{p,\phi,\cF}}$ by the definition.
If $\phi$ is almost increasing, 
then $\|f\|_{\cL_{p,\phi}}$ and $\|f\|_{\cL_{p,\phi,\cF}}$ are equivalent.
Actually, for any $A\in \cF_n$, 
there exists a sequence of atoms $B_{\ell}\in A(\cF_n)$, $\ell=1,2,\cdots$, 
such that $A=\cup_{\ell}B_{\ell}$ and $P(A)=\sum_{\ell}P(B_{\ell})$.
Then
\begin{align*}
 \int_A|f-E_nf|^p\,dP
 &=
 \sum_{\ell}\int_{B_{\ell}}|f-E_nf|^p\,dP \\
 &\le
 \sum_{\ell}\phi(P(B_{\ell}))^pP(B_{\ell})\|f\|_{\cL_{p,\phi}}^p \\
 &\le C^p
 \phi(P(A))^pP(A)\|f\|_{\cL_{p,\phi}}^p.
\end{align*}
This shows $\|f\|_{\cL_{p,\phi,\cF}}\le C\|f\|_{\cL_{p,\phi}}$.
If $\phi$ is not almost increasing, then 
$\|f\|_{\cL_{p,\phi}}$ is not equivalent to $\|f\|_{\cL_{p,\phi,\cF}}$ in general,
see \cite{Nakai-Sadasue2012JFSA}.
The norm \eqref{norm F} was introduced by \cite{MiNaSa2012MathNachr}
for general $\{\cF_n\}_{n\ge0}$.
\end{enumerate}
\end{rem}

By Theorem~\ref{thm:PWM} we have the next two corollaries immediately:

\begin{cor}\label{cor:PWM BMO}
Let $\{\cF_n\}_{n\ge0}$ be regular and $\cF_0=\{\emptyset,\Omega\}$.
Then
\begin{equation*}
  \pwm(\BMON)=\cL_{1,\psi}\cap L_{\infty},
\end{equation*}
where 
$\psi(r)=1/\log(e/r)$.
Moreover, for $g\in\pwm(\BMON)$, $\|g\|_{Op}$ is equivalent to 
$\|g\|_{\cL_{1,\psi}}+\|g\|_{L_{\infty}}$.
\end{cor}

\begin{cor}\label{cor:PWM Lip}
Let $\{\cF_n\}_{n\ge0}$ be regular, $\cF_0=\{\emptyset,\Omega\}$ and $\alpha>0$.
Then
\begin{equation*}
  \pwm(\LipN_{\alpha})=\Lip_{\alpha}\cap L_{\infty}.
\end{equation*}
Moreover, for $g\in\pwm(\LipN_{\alpha})$, $\|g\|_{Op}$ is equivalent to 
$\|g\|_{\Lip_{\alpha}}+\|g\|_{L_{\infty}}$.
\end{cor}

\begin{exmp}\label{exmp:PWM BMO}
Let $\{\cF_n\}_{n\ge0}$, $p$ and $\phi$ satisfy the assumption in Theorem~\ref{thm:PWM}.
For a sequence 
$$
 B_0\supset B_1\supset\cdots\supset B_n\supset\cdots, 
 \quad 
 B_n\in A(\cF_n),
$$
let 
\begin{equation}
 g=\sin h,
 \quad\text{where}\quad
 h=\sum_{n=1}^{\infty}
  \frac{\phi(P(B_n))}{\phi_*(P(B_n))}
   \left(\frac{P(B_{n-1})}{P(B_n)}\chi_{B_n}-\chi_{B_{n-1}}\right).
\end{equation}
Then $h$ is in $\cL_{p,\phi/\phi_*}$,
see Lemma~\ref{lem:fBn} and Remark~\ref{rem:fBn}.
Hence $g\in\cL_{p,\phi/\phi_*}\cap L_{\infty}$,
since $\sin\theta$ is Lipschitz continuous, see Remark~\ref{rem:LipC}.
That is, $g\in\pwm(\cLN_{p,\phi})$.
If $\phi\equiv1$, then $\phi(r)/\phi_*(r)=1/\log(e/r)$ and $g\in\pwm(\BMON)$. 
\end{exmp}

\begin{exmp}\label{exmp:phi}
The following function
satisfies the doubling condition and the property \eqref{int phi}:
\begin{equation*}
 \phi(r)=r^{\alpha}(\log(e/r))^{-\beta} 
 \quad (\alpha\in(-1/p,\infty), \ \beta\in(-\infty,\infty)).
\end{equation*}
If $\alpha\in(-1/p,0)$ and $\beta\in(-\infty,\infty)$, 
then $\phi_*\sim\phi$, that is,
there exists a positive constant $C$ such that
$C^{-1}\phi(r)\le\phi_*(r)\le C\phi(r)$ for all $r\in(0,1]$.
In general, under the assumption of Theorem~\ref{thm:PWM},
if $\phi_*\sim\phi$, then $\cL_{1,\phi/\phi_*}=\BMO$ and then
\begin{equation*}
  \pwm(\cLN_{p,\phi})=\BMO\cap L_{\infty}=L_{\infty}.
\end{equation*}
If $\alpha\in[0,\infty)$ and $\beta\in(-\infty,\infty)$, 
then $\phi_*\not\sim\phi$ and $\phi(r)/\phi_*(r)\to0$ as $r\to0$.
In this case $\cL_{1,\phi/\phi_*}\cap L_{\infty}\ne L_{\infty}$ in general (see also Remark~\ref{rem:dyadic}).
In particular, 
if $\alpha=0$ and $\beta\in(1,\infty)$, 
or if $\alpha\in(0,\infty)$ and $\beta\in(-\infty,\infty)$, 
then $\phi_*\sim1$.
In general, under the assumption of Theorem~\ref{thm:PWM},
if $\phi_*\sim1$, 
then $\cL_{1,\phi/\phi_*}=\cL_{1,\phi}\subset L_{\infty}$ 
by Lemma~\ref{lem:fB} below, 
and then
\begin{equation*}
  \pwm(\cLN_{p,\phi})=\cL_{1,\phi}\cap L_{\infty}=\cLN_{1,\phi}. 
\end{equation*}
Moreover, if $\phi$ is almost increasing, 
then we can use the John-Nirenberg type inequality in
\cite[Theorem~2.9]{MiNaSa2012MathNachr}, that is,
\begin{equation*}
  \pwm(\cLN_{p,\phi})=\cLN_{p,\phi}.
\end{equation*}
We can also take the function
$$
 \phi(r)=r^{\alpha}(\log(e/r))^{-\beta}(\log\log(e/r))^{-\gamma}
 \quad (\alpha\in(-1/p,\infty), \ \beta,\gamma\in(-\infty,\infty)),
$$
and so on.
\end{exmp}

Next, for a martingale $\sqc{f}$ relative to $\{\cF_n\}_{n\ge0}$, 
it is said to be $\cL_{p,\lambda}$-bounded 
if $f_n\in \cL_{p,\lambda}$ $(n\ge0)$ 
and $\sup_{n\ge0}\|f_n\|_{\cL_{p,\lambda}}<\infty$.
Similarly, 
the martingale $\sqc{f}$ is said to be $\cLN_{p,\lambda}$-bounded 
if $f_n\in \cLN_{p,\lambda}$ $(n\ge0)$ 
and $\sup_{n\ge0}\|f_n\|_{\cLN_{p,\lambda}}<\infty$.

Let 
\begin{equation*}
  \cL_{p,\phi}(\cF_n)
  = \{f\in L_1: \text{$f$ is $\cF_n$-measurable and } \|f\|_{\cL_{p,\phi}}<\infty \}
\end{equation*}
and
\begin{equation*}
  \cLN_{p,\phi}(\cF_n)
  = \{f\in L_1: \text{$f$ is $\cF_n$-measurable and } \|f\|_{\cLN_{p,\phi}}<\infty \}.
\end{equation*}

Then we have the following:

\begin{thm}\label{thm:PWM M}
Let $\{\cF_n\}_{n\ge0}$ be regular, $\cF_0=\{\emptyset,\Omega\}$, 
$p\in[1,\infty)$ and $\phi:(0,1]\to(0,\infty)$.
Assume that $\phi$ satisfies the doubling condition
and \eqref{int phi}.
Let $g\in L_1$ and $\sqc{g}$ be its corresponding martingale 
with $g_n=E_ng$ $(n\ge0)$.
If $g\in\pwm(\cLN_{p,\phi})$, then $g_n\in\pwm(\cLN_{p,\phi}(\cF_n))$.
Conversely,
if $g_n\in\pwm(\cLN_{p,\phi}(\cF_n))$ and $\sup_{n\ge0}\|g_n\|_{Op}<\infty$,
then $g\in\pwm(\cLN_{p,\phi})$.
\end{thm}

We show several lemmas in Section~\ref{sec:lemmas} to prove Theorem~\ref{thm:PWM}
in Section~\ref{sec:proof}.
We prove Theorem~\ref{thm:PWM M}
in Section~\ref{sec:pc}.

At the end of this section, we make some conventions. 
Throughout this paper, we always use $C$ to denote a positive constant 
that is independent of the main parameters involved 
but whose value may differ from line to line.
Constants with subscripts, such as $C_p$, is dependent on the subscripts.
If $f\le Cg$, we then write $f\ls g$ or $g\gs
f$; and if $f \ls g\ls f$, we then write $f\sim g$.

\section{Lemmas}\label{sec:lemmas} 

To prove Theorem~\ref{thm:PWM} we show several lemmas in this section.
The first lemma was proved in \cite{Nakai-Sadasue2012JFSA}.
\begin{lem}[{\cite[Lemma~3.3]{Nakai-Sadasue2012JFSA}}]\label{lem:NS3.3}
Let $\{\cF_n\}_{n\ge0}$ be regular.
Then every sequence
$$
 B_0\supset B_1\supset\cdots\supset B_n\supset\cdots, 
 \quad 
 B_n\in A(\cF_n)
$$
has the following property; for each $n\ge1$,
$$
 B_n=B_{n-1}\ \text{or}\ 
 \left(1+\frac1R\right)P(B_n)\le P(B_{n-1})\le R P(B_n),
$$
where $R$ is the constant in \eqref{regular}.
\end{lem}

For a function $f\in L_1$ and an atom $B\in A(\cF_n)$, 
let
$$
 f_B=\frac1{P(B)}\int_B f\,dP.
$$
For a function $\phi:(0,1]\to(0,\infty)$, 
let $\phi_*$ be defined by \eqref{phi_*}.
If $\phi$ satisfies the doubling condition,
then $\phi(r)\le C\phi_*(r)$ for all $r\in(0,1]$.

\begin{lem}\label{lem:fB}
Let $\{\cF_n\}_{n\ge0}$ be regular, $\cF_0=\{\emptyset,\Omega\}$, 
$p\in[1,\infty)$ and $\phi:(0,1]\to(0,\infty)$.
Assume that $\phi$ satisfies the doubling condition.
For $f\in\cLN_{p,\phi}$ and $B\in\cup_{n\ge0}A(\cF_n)$, 
\begin{equation}\label{fB}
 |f_B|\le C\phi_*(P(B))\|f\|_{\cLN_{p,\phi}}.
\end{equation}
\end{lem}
\begin{proof}
By Lemma~\ref{lem:NS3.3}, we can choose $B_{k_j}\in A(\cF_{k_j})$, $0=k_0<k_1<\cdots<k_m\le n$, 
such that $B_{k_0}\supset B_{k_1}\supset B_{k_2} \supset\cdots\supset B_{k_m}=B$
and that $(1+1/R)P(B_{k_j})\le P(B_{k_{j-1}})\le RP(B_{k_j})$.
Then, we have
\begin{align*}
 |f_{B_{k_j}}-f_{B_{k_{j-1}}}|
 &=
 \left|\frac1{P(B_{k_j})}\int_{B_{k_j}}f(\omega)\,dP-\frac1{P(B_{k_{j-1}})}\int_{B_{k_{j-1}}}f(\omega)\,dP \right|\\
 &=
 \left|\frac1{P(B_{k_j})}\int_{B_{k_j}}[f-E_{k_{j-1}}f](\omega)\,dP\right|\\
 &\le
 \left(\frac1{P(B_{k_j})}\int_{B_{k_j}}|f-E_{k_{j-1}}f|^p\,dP\right)^{1/p} \\
 &\ls
 \left(\frac1{P(B_{k_{j-1}})}\int_{B_{k_{j-1}}}|f-E_{k_{j-1}}f|^p\,dP\right)^{1/p} \\
 &\le
 \phi (P(B_{k_{j-1}}))\ \|f\|_{\cLN_{p,\phi}}. 
\end{align*}
Since $\phi$ satisfies the doubling condition, 
\begin{align*}
 |f_{B}-f_{B_0}|
 &\le
 \sum_{j=1}^m |f_{B_{k_j}}-f_{B_{k_{j-1}}}|\\
 &\ls
 \sum_{j=1}^m \phi (P(B_{k_{j-1}}))\ \|f\|_{\cLN_{p,\phi}}\\
 &\ls
 \sum_{j=1}^m \int_{P(B_{k_{j}})}^{P(B_{k_{j-1}})}\frac{\phi (t)}{t}\,dt\ \|f\|_{\cLN_{p,\phi}}\\
 &=
 \int_{P(B)}^{1}\frac{\phi (t)}{t}\,dt\ \|f\|_{\cLN_{p,\phi}}\\
 &=
 \{\phi_*(P(B))-1\}\, \|f\|_{\cLN_{p,\phi}}.
\end{align*}
On the other hand, 
$$
 |f_{B_0}|=|Ef|\le\|f\|_{\cLN_{p,\phi}}.
$$
Therefore, we have \eqref{fB}. 
\end{proof}

\begin{lem}\label{lem:chiB}
Let $\cF_0=\{\emptyset,\Omega\}$, $p\in[1,\infty)$ and $\phi:(0,1]\to(0,\infty)$.
Assume 
that $r\phi(r)^p$ is almost increasing.
For any atom $B\in\cup_{n\ge0}A(\cF_n)$, 
the characteristic function $\chi_B$ is in $\cLN_{p,\phi}$ and 
there exists a positive constant $C$, independent of $B$, such that
\begin{equation}\label{eqn:chiB}
 \|\chi_B\|_{\cLN_{p,\phi}}\le \frac{C}{\phi(P(B))}.
\end{equation}
\end{lem}
\begin{proof}
Let $B\in A(\cF_n)$ and $B'\in A(\cF_k)$. Let $B_j\in A(\cF_j)$, $0\le j\le n$, 
such that 
$B_0\supset B_1\supset\cdots\supset B_n=B$.

If $k\geq n$, then $\chi_B-E_k\chi_B=0$ and 
$$
  \int_{B'}|\chi_B-E_k\chi_B|^p\, dP=0.
$$
If $k<n$ and $B'\not=B_k$, then $B'\cap B_k=\emptyset$ and 
$$
  \int_{B'}|\chi_B-E_k\chi_B|^p\, dP=0.
$$
Hence, we have 
$$
  \|\chi_B\|_{\cL_{p,\phi}}=\sup_{k<n}\frac{1}{\phi (P(B_k))}\left(\frac{1}{P(B_k)}\int_{B_k}|\chi_B-E_k\chi_B|^p\, dP\right)^{1/p}.
$$
For $k<n$, since $r\phi(r)^p$ is almost increasing,
\begin{align*}
  &
 \frac{1}{\phi (P(B_k))^p}\frac{1}{P(B_k)}\int_{B_k}|\chi_B-E_k\chi_B|^p\, dP\\
  &=
 \frac{1}{\phi (P(B_k))^pP(B_k)} \\
  &\phantom{****}\times
 \left\{
 P(B_n)\left(1-\frac{P(B_n)}{P(B_k)}\right)^p+(P(B_k)-P(B_n))\left(\frac{P(B_n)}{P(B_k)}\right)^p\right\} \\
  &\ls
 \frac{1}{\phi (P(B_n))^pP(B_n)} \\
  &\phantom{****}\times
 \left\{
 P(B_n)\left(1-\frac{P(B_n)}{P(B_k)}\right)^p+(P(B_k)-P(B_n))\left(\frac{P(B_n)}{P(B_k)}\right)^p\right\} \\
  &=
 \frac{1}{\phi (P(B_n))^p}\left\{
 \left(1-\frac{P(B_n)}{P(B_k)}\right)^p+\left(1-\frac{P(B_n)}{P(B_k)}\right)\left(\frac{P(B_n)}{P(B_k)}\right)^{p-1}\right\} \\
  &\ls
 \frac{1}{\phi (P(B_n))^p}= \frac{1}{\phi (P(B))^p}.
\end{align*}
Therefore, we have 
\begin{equation}\label{chiB phi}
 \|\chi_B\|_{\cL_{p,\phi}}\ls \frac{1}{\phi (P(B))}.
\end{equation}
On the other hand, since $r\phi(r)^p$ is almost increasing, 
\begin{equation}\label{chiB 0}
 |E\chi_B|
 =P(B)\le P(B)^{1/p}
 \ls \frac{1}{\phi (P(B))}.
\end{equation}
Combining \eqref{chiB phi} and \eqref{chiB 0}, we have \eqref{eqn:chiB}.
\end{proof}

\begin{lem}\label{lem:fBn}
Let $\{\cF_n\}_{n\ge0}$ be regular, $\cF_0=\{\emptyset,\Omega\}$, 
$p\in[1,\infty)$ and $\phi:(0,1]\to(0,\infty)$.
Assume that $\phi$ satisfies the doubling condition and \eqref{int phi}.
For a sequence 
$$
 B_0\supset B_1\supset\cdots\supset B_n\supset\cdots, 
 \quad 
 B_n\in A(\cF_n),
$$
let 
$$
 f_0=\chi_{B_0},
\quad
 u_k=
  \phi(P(B_k))\left(\frac{P(B_{k-1})}{P(B_k)}\chi_{B_k}-\chi_{B_{k-1}}\right),
$$
and let
\begin{equation}
 f_n=f_0
 +\sum_{k=1}^{n} u_k.
\end{equation}
Then $\sqc{f}$ is a martingale and $\cLN_{p,\phi}$-bounded.
The sum $f\equiv f_0+\sum_{k=1}^{\infty}u_k$ converges a.s. and in $L_p$, and
$E_nf=f_n$ for $n\ge0$.
Moreover, there exist positive constants $C_1$ and $C_2$, 
independent of the sequence of atoms,
such that
\begin{equation}
 \|f\|_{\cLN_{p,\phi}}\le C_1
 \quad\text{and}\quad
 |f_{B_n}|\ge C_2\phi_*(P(B_n)), \ n\ge0.
\end{equation}
\end{lem}

\begin{proof}
Since $E_n[u_k]=0$ for $k>n$, $\sqc{f}$ is a martingale. 
We show that the sum $f_0+\sum_{k=1}^{\infty}u_k$ converges in $L_p$. 
If $\lim_{k\to\infty}P(B_k)>0$
then the convergence is clear because 
there exists $m$ such that $B_m=B_n$ for all $n\ge m$. 
We assume that $\lim_{k\to\infty}P(B_k)=0$. 
By Lemma~\ref{lem:NS3.3}, 
we can take a sequence of integers 
$0=k_0<k_1<\cdots<k_j<\cdots$ that satisfies 
\begin{equation}\label{ball}
(1+1/R)P(B_{k_j})\le P(B_{k_{j-1}})\le RP(B_{k_j}),
\end{equation}
and $B_{k_{j-1}}=B_{k}$ if $k_{j-1}\le k< k_j$.
In this case we can write
$$
 f_n
 =
 \chi_{B_0}
 +\sum_{1\le k_j\le n} \phi (P(B_{k_j}))
  \left(\frac{P(B_{k_{j-1})}}{P(B_{k_j})}\chi_{B_{k_j}}-\chi_{B_{k_{j-1}}}\right).
$$
Note that, by Remark~\ref{rem:phi} and \cite[Lemma~7.1]{Nakai2008AMSin}, 
the doubling condition and \eqref{int phi} implies
\begin{equation}\label{int phi 1/p}
 \int_0^r\phi(t)t^{1/p-1}\,dt\le C_p\phi(r)r^{1/p}
 \quad\text{for all $r\in(0,1]$}.
\end{equation}
Using the doubling condition and \eqref{int phi 1/p}, we have
\begin{align}\label{absolute}
&
\sum_{k_j>n}
\phi (P(B_{k_j}))
\left\|\frac{P(B_{k_{j-1})}}{P(B_{k_j})}\chi_{B_{k_j}}-\chi_{B_{k_{j-1}}}\right\|_{L_p} 
\\
&\le
\sum_{k_j>n}
\phi (P(B_{k_j}))
(R\|\chi_{B_{k_j}}\|_{L_p}+\|\chi_{B_{k_{j-1}}}\|_{L_p}) \notag \\
&\le
2R\sum_{k_j>n}\phi (P(B_{k_j})) P(B_{k_j})^{1/p}\notag \\
&\le
C\sum_{k_j>n}\int_{P(B_{k_j})}^{P(B_{k_{j-1}})}\phi(t)t^{1/p-1}\,dt\notag \\
&\le
C\int_0^{P(B_{n})} \phi(t)t^{1/p-1}\,dt\notag \\
&\le
CC_p\phi (P(B_{n}))P(B_{n})^{1/p}. \notag
\end{align}
We can deduce from \eqref{absolute} that 
$f\equiv f_0+\sum_{k=1}^{\infty}u_k$ converges in $L_p$. 
By the martingale convergence theorem, $f_0+\sum_{k=1}^{\infty}u_k$ also converges almost surely. 
Moreover, we have $E_nf=f_n$ and
\begin{equation}\label{onBn}
\left(\frac1{P(B_n)}\int_{B_n} |f-E_nf|^p\,dP\right)^{1/p} 
\le
CC_p\phi (P(B_n)).
\end{equation}
For $B'\in A(\cF_n)$, we have
\begin{equation}\label{support}
(f-E_nf)\chi_{B'}=
\begin{cases}
 f-E_nf & (B' = B_n) \\
 \displaystyle
 0& (B' \not= B_n).
\end{cases}
\end{equation}
Combining \eqref{onBn} and \eqref{support}, we have
$
 \|f\|_{\cL_{p,\phi}}\le C
$
where $C$ is a positive constant independent of the sequence of atoms. 
Moreover, since $B_0=\Omega$,
$$
 |Ef|=|f_0|=1.
$$
Therefore, 
$
 \|f\|_{\cLN_{p,\phi}}\le C_1 
$ 
where $C_1$ is a positive constant independent of the sequence of atoms.

We now show $|f_{B_n}|\ge C_2\phi_*(P(B_n))$. On the atom $B_n$, we have
\begin{equation*}
 f_n=1+\sum_{1\le k_j\le n}
  \phi(P(B_{k_j}))\left(\frac{P(B_{k_{j-1}})}{P(B_{k_j})}-1\right)
 \ge
  1+\frac1{R} \sum_{1\le k_j\le n} \phi(P(B_{k_j})).
\end{equation*}
Therefore, we have 
\begin{align*}
 |f_{B_n}|
  &=
 \left|\frac1{P(B_n)}\int_{B_n}f_n\, dP\right|
\\
  &\ge
   1+\frac1{R} \sum_{1\le k_j\le n} \phi(P(B_{k_j}))
\\
 &\sim
  1+\sum_{1\le k_j\le n}\int_{P(B_{k_j})}^{P(B_{k_{j-1}})}\frac{\phi(t)}{t}\,dt 
\\
  &=
   1+\int_{P(B_n)}^{1}\frac{\phi(t)}{t}\, dt
  =
 \phi_*(P(B_n))
\end{align*}
That is, 
$
 |f_{B_n}|\ge C_2\phi_*(P(B_n))
$
where $C_2$ is a positive constant independent of the sequence of atoms.
\end{proof}

\begin{rem}\label{rem:fBn}
From the proof of Lemma~\ref{lem:fBn} we see that, for
\begin{equation}\label{h}
 h=\sum_{k=1}^{\infty} u_k,
 \quad
 h_0=0,\quad
 h_n=\sum_{k=1}^{n} u_k \ (n\ge1), 
\end{equation}
$h$ is in $\cL_{p,\phi}$ and 
$\sqc{h}$ is its corresponding martingale with $h_n=E_nh$ $(n\ge0)$.
\end{rem}

\begin{rem}\label{rem:dyadic}
Let $(\Omega,\cF,P)$ be as follows:
\begin{align*}
 &
\Omega=[0,1), \quad 
 A(\cF_n)=\{I_{n,j}=[j2^{-n},(j+1)2^{-n}):j=0,1,\cdots,2^n-1\}\\
&
\cF_n=\sigma(A(\cF_n)), \quad\cF=\sigma(\cup_{n}\cF_n), \quad P =\text{ the Lebesgue measure}.
\end{align*}
If $\phi(r)=1/\log(e/r)$,
then $h$ in \eqref{h} is unbounded. 
Actually, 
$$
 u_k=\frac1{1+k\log2}(2\chi_{B_k}-\chi_{B_{k-1}}),
$$
and
$$
 h=\sum_{k=1}^n\frac1{1+k\log2}-\frac1{1+(n+1)\log2}
 \quad\text{on $B_n\setminus B_{n+1}$}.
$$
\end{rem}

\begin{rem}\label{rem:LipC}
If $F:\C\to\C$ is Lipschitz continuous, that is, 
$$
 |F(z_1)-F(z_2)|\le C|z_1-z_2|, \quad z_1,z_2\in\C,
$$
then, for $B\in\cF_n$,
$$
 \int_{B}|F(f)-E_n[F(f)]|\,dP
 \le
 2C
 \int_{B}|f-E_nf|\,dP.
$$
Actually,
\begin{align*}
 &\int_{B}|F(f)-E_n[F(f)]|\,dP \\
 &\le
 \int_{B}|F(f)-F(E_nf)|\,dP
 +\int_{B}|F(E_nf)-E_n[F(f)]|\,dP \\
 &=
 \int_{B}|F(f)-F(E_nf)|\,dP
 +\int_{B}|E_n[F(E_nf)-F(f)]|\,dP \\
 &\le 
 2\int_{B}|F(f)-F(E_nf)|\,dP \\
 &\le 
 2C\int_{B}|f-E_nf|\,dP.
\end{align*}
\end{rem}

\begin{lem}\label{lem:NY3.3}
Let $p\in[1,\infty)$ and $\phi:(0,1]\to(0,\infty)$.
Suppose that $f\in\cL_{p,\phi}$ and $g\in L_{\infty}$.
Then $fg\in\cL_{p,\phi}$ if and only if 
\begin{equation}\label{eqn:Ffg}
  F(f,g)
  :=
  \sup_{n\ge0}\sup_{B\in A(\cF_n)}
  \frac{|f_B|}{\phi(P(B))}\left(\frac{1}{P(B)}\int_B|g-E_ng|^p\,dP\right)^{1/p}
  <\infty.
\end{equation}
In this case,
\begin{equation}\label{Ffg}
  \left|F(f,g)-\|fg\|_{\cL_{p,\phi}}\right|
  \le
  2\|f\|_{\cL_{p,\phi}}\|g\|_{L_{\infty}}.
\end{equation}
\end{lem}

\begin{proof}
Let $f\in\cL_{p,\phi}$ and $g\in L_{\infty}$. Let $B\in A(\cF_n)$. 
Since $E_nf=f_B$ on $B$, we can use the same method as in 
\cite[Lemma~3.5]{Nakai1993Studia} and we have
\begin{align}\label{nakai1993}
  &
 \left|
 \left(\frac1{P(B)}\int_B|fg-E_n[fg]|^p\, dP\right)^{1/p}-
 |f_B|\left(\frac1{P(B)}\int_B|g-E_ng|^p\, dP\right)^{1/p}
 \right|\\
  &\le
 2\left(\frac1{P(B)}\int_B|(f-E_nf)g|^p\, dP\right)^{1/p}
  \le
 2\phi (P(B))\|f\|_{\cL_{p,\phi}}\|g\|_{L_{\infty}}.\notag
\end{align}
Therefore, $fg\in\cL_{p,\phi}$ if and only if $F(f,g)<\infty$. 
In this case, we can deduce \eqref{Ffg} from \eqref{nakai1993}. 
\end{proof}

\begin{lem}\label{lem:NY3.4}
Let $\{\cF_n\}_{n\ge0}$ be regular, $\cF_0=\{\emptyset,\Omega\}$, 
$p\in[1,\infty)$ and $\phi:(0,1]\to(0,\infty)$.
Assume that $r\phi(r)^p$ is almost increasing
and that $\phi$ satisfies the doubling condition. 
If $g\in\pwm(\cLN_{p,\phi})$, 
then $g\in L_{\infty}$ and $\|g\|_{L_{\infty}}\le C\|g\|_{Op}$
for some positive constant $C$ independent of $g$.
\end{lem}

\begin{proof}
Let $g\in\pwm(\cLN_{p,\phi})$. 
Since the constant function $1$ is in $\cLN_{p,\phi}$,
the pointwise multiplication $g=g\cdot1$ is in $\cLN_{p,\phi}$,
which implies $g\in L_1$.
Then
\begin{equation*}
  E[|g|]
  \le E[|g-Eg|]+|Eg|
  \le \max(1,\phi(1))\|g\|_{\cLN_{p,\phi}}
  \ls\|g\|_{Op}\|1\|_{\cLN_{p,\phi}}
  =\|g\|_{Op}.
\end{equation*}
Since $\{\cF_n\}_{n\ge0}$ is regular, we also have $E_ng\in L_{\infty}$ as follows:
$$
  E_n[|g|]\le RE_{n-1}[|g|]\le \cdots \le R^nE_0[|g|]=R^nE[|g|].
$$

Next we shall show that there exists a positive constant $C$ such that
$\|g\|_{L_{\infty}}\le C\|g\|_{Op}$.
Then we have the conclusion.
Let $B\in A(\cF_n)$ such that $|g_B|\geq \|E_ng\|_{L_{\infty}}/2$. 
By Lemma~\ref{lem:NS3.3} 
there exists $B'\in A(\cF_{n'})$ with $B\subset B'$ such that 
$(1+1/R)P(B)\le P(B')\le RP(B)$. 
Then, we have 
\begin{align*}
 \|g\chi_B\|_{\cLN_{p,\phi}}
  &\ge
 \frac{1}{\phi (P(B'))}\left(\frac1{P(B')}
 \int_{B'}|g\chi_B-E_{n'}[g\chi_B]|^p\,dP\right)^{1/p}\\
  &\ge
 \frac{1}{\phi (P(B'))}\left(\frac1{P(B')}
 \int_{B'\setminus B}|g\chi_B-E_{n'}[g\chi_B]|^p\,dP\right)^{1/p}\\
  &=
 \frac{1}{\phi (P(B'))}\left(\frac1{P(B')}
 \int_{B'\setminus B}|E_{n'}[[E_ng]\chi_B]|^p\,dP\right)^{1/p}.
\end{align*}
Since 
$|[E_ng]\chi_B|=|g_B\chi_B|\ge \|E_ng\|_{L_{\infty}}\chi_B/2$, 
we have
\begin{equation*}
 \int_{B'\setminus B}|E_{n'}[[E_ng]\chi_B]|^p\,dP
 \ge
 \left(\frac{\|E_ng\|_{L_{\infty}}}{2}\right)^p
 \left(\frac{P(B)}{P(B')}\right)^pP(B'\setminus B).
\end{equation*}
Hence, we have
\begin{equation}\label{eqn:op}
 \|g\chi_B\|_{\cLN_{p,\phi}}
 \ge
  \frac{\|E_ng\|_{L_{\infty}}}{2R(R+1)^{1/p}\phi (P(B'))}.
\end{equation}
Using \eqref{eqn:op}, Lemma~\ref{lem:chiB} 
and the doubling condition on $\phi$, we have
\begin{align*}
 \|E_ng\|_{L_{\infty}}
  &\le
 2R(R+1)^{1/p}\phi (P(B'))
 \|g\chi_B\|_{\cLN_{p,\phi}} \\
  &\ls
 \|g\|_{Op}\frac{\phi (P(B'))}{\phi(P(B))}\\
  &\ls
 \|g\|_{Op}.
\end{align*}
Therefore, 
$$
 \|g\|_{L_{\infty}}
 =\sup_{n\ge 0}\|E_ng\|_{L_{\infty}}\le C\|g\|_{Op}.
$$
This shows the conclusion.
\end{proof}

\section{Proof of Theorem~\ref{thm:PWM}}\label{sec:proof}

We first show that 
\begin{equation}\label{suff}
  \cL_{p,\phi/\phi_*}\cap L_{\infty}\subset \pwm(\cLN_{p,\phi})
 \quad \text{and}\quad 
  \|g\|_{Op}\le C(\|g\|_{\cL_{p,\phi/\phi_*}}+\|g\|_{L_{\infty}}).
\end{equation}
Let $g\in \cL_{p,\phi/\phi_*}\cap L_{\infty}$ and $f\in \cLN_{p,\phi}$. 
Let $F(f,g)$ be as in Lemma~\ref{lem:NY3.3}.
Then, by the definition of $F(f,g)$ and Lemma~\ref{lem:fB} we have
$$
 F(f,g)\le C\|f\|_{\cLN_{p,\phi}}\|g\|_{\cL_{p,\phi/\phi_*}}<\infty .
$$
Therefore, by Lemma~\ref{lem:NY3.3}, we have $fg\in\cL_{p,\phi}$ and
\begin{equation}\label{fg}
 \|fg\|_{\cL_{p,\phi}}\le C\|f\|_{\cLN_{p,\phi}}\|g\|_{\cL_{p,\phi/\phi_*}}
 +2\|f\|_{\cL_{p,\phi}}\|g\|_{L_{\infty}}.
\end{equation}
On the other hand, we have 
\begin{equation}\label{fg0}
 |E[fg]|\le \|g\|_{L_{\infty}}E[|f|]
 \le \|g\|_{L_{\infty}}\max(1,\phi(1))\|f\|_{\cLN_{p,\phi}}.
\end{equation}
Combining \eqref{fg} and \eqref{fg0}, we obtain \eqref{suff}.

We now show the converse, that is,
\begin{equation}\label{ness}
  \pwm(\cLN_{p,\phi})\subset \cL_{p,\phi/\phi_*}\cap L_{\infty} 
  \quad \text{and}\quad 
  \|g\|_{\cL_{p,\phi/\phi_*}}+\|g\|_{L_{\infty}}\le C \|g\|_{Op}.
\end{equation}
Let $g\in \pwm(\cLN_{p,\phi})$. 
By Lemma~\ref{lem:NY3.4}, 
we have $g\in L_{\infty}$ and $\|g\|_{L_{\infty}}\le C \|g\|_{Op}$. 
Let $B\in A(\cF_n)$. We take $B_j\in A(\cF_j)$ with $B_n=B$ such that
$$
 B_0\supset B_1\supset\cdots\supset B_n\supset\cdots .
$$
Let $f$ be the function described in Lemma~\ref{lem:fBn}. 
Then, combining Lemma~\ref{lem:fBn} and Lemma~\ref{lem:NY3.3}, we have 
\begin{align*}
 &\frac{C_2\phi_*(P(B))}{\phi (P(B))}
 \left(\frac1{P(B)}\int_B|g-E_ng|^p\, dP\right)^{1/p} \\
  &\le
 \frac{|f_B|}{\phi (P(B))}\left(\frac1{P(B)}\int_B|g-E_ng|^p\, dP\right)^{1/p}\\
  &\le
   F(f,g)\\
  &\le
   \|fg\|_{\cL_{p,\phi}}+2\|g\|_{L_{\infty}}\|f\|_{\cL_{p,\phi}}\\
  &\le
   \|g\|_{Op}\|f\|_{\cLN_{p,\phi}}+2C\|g\|_{Op}\|f\|_{\cL_{p,\phi}}\\
  &\ls
   \|g\|_{Op}\|f\|_{\cLN_{p,\phi}}
  \le
   C_1\|g\|_{Op}.
\end{align*}
Therefore, we have \eqref{ness}.

\section{Proof of Theorem~\ref{thm:PWM M}}\label{sec:pc} 

To prove Theorem~\ref{thm:PWM M} we use the following proposition.
It can be shown by the same way as \cite[Proposition 2.2]{Nakai-Sadasue2012JFSA}
which deals with the case $\phi(r)=r^{\lambda}$, $\lambda\in(-\infty,\infty)$.

\begin{prop}\label{prop:M}
Let $1\le p<\infty$ and $\phi:(0,1]\to(0,\infty)$.
Let $f\in L_1$ and $\sqc{f}$ be its corresponding martingale with $f_n=E_nf$ $(n\ge0)$.
\begin{enumerate}
\item
If $f\in \cL_{p,\phi}$,  
then $\sqc{f}$ is $\cL_{p,\phi}$-bounded and
\begin{equation*}
 \|f\|_{\cL_{p,\phi}}\ge\sup_{n\ge0}\|f_n\|_{\cL_{p,\phi}}.
\end{equation*}
Conversely, if $\sqc{f}$ is $\cL_{p,\phi}$-bounded,
then $f\in \cL_{p,\phi}$ and 
\begin{equation*}
 \|f\|_{\cL_{p,\phi}}\le\sup_{n\ge0}\|f_n\|_{\cL_{p,\phi}}.
\end{equation*}
\item
If $f\in \cLN_{p,\phi}$,  
then $\sqc{f}$ is $\cLN_{p,\phi}$-bounded and
\begin{equation*}
 \|f\|_{\cLN_{p,\phi}}\ge\sup_{n\ge0}\|f_n\|_{\cLN_{p,\phi}}.
\end{equation*}
Conversely, if $\sqc{f}$ is $\cLN_{p,\phi}$-bounded,
then $f\in \cLN_{p,\phi}$ and 
\begin{equation*}
 \|f\|_{\cLN_{p,\phi}}\le\sup_{n\ge0}\|f_n\|_{\cLN_{p,\phi}}.
\end{equation*}
\end{enumerate}
\end{prop}

\begin{rem}\label{rem:M}
In general, for $f\in \cL_{p,\phi}\cap L_{1,0}$ (res. $f\in\cLN_{p,\phi}$), 
its corresponding martingale $\sqc{f}$ with $f_n=E_nf$ 
does not always converge to $f$ in $\cL_{p,\phi}$ (res. $\cLN_{p,\phi}$).
See Remark~3.7 in \cite{Nakai-Sadasue2012JFSA} for the case $\phi(r)=r^{\lambda}$.
\end{rem}

\begin{proof}[Proof of Theorem~\ref{thm:PWM M}]
Let $g\in\pwm(\cLN_{p,\phi})$ and $f\in\cLN_{p,\phi}(\cF_n)$. 
Then, using Proposition~\ref{prop:M}, we have
$$
  \|E_n[g] f\|_{\cLN_{p,\phi}}=\|E_n[gf]\|_{\cLN_{p,\phi}}
 \le 
  \|gf\|_{\cLN_{p,\phi}}\le \|g\|_{Op} \| f\|_{\cLN_{p,\phi}}.
$$
Therefore, we have $E_ng\in\pwm(\cLN_{p,\phi}(\cF_n))$.

Conversely, assume that $E_ng\in\pwm(\cLN_{p,\phi}(\cF_n))$ and 
$\sup_{n\ge0}\|E_ng\|_{Op}<\infty$. 
Then, using Proposition~\ref{prop:M} and Theorem~\ref{thm:PWM}, we have
$$
  \|g\|_{\cL_{p,\phi/\phi_*}}+\|g\|_{L_{\infty}}
 \le
  \sup_{n\ge0}\|E_ng\|_{\cL_{p,\phi/\phi_*}}+\sup_{n\ge0}\|E_ng\|_{L_{\infty}}
 \ls
  \sup_{n\ge0}\|E_ng\|_{Op}<\infty .
$$
Using Theorem~\ref{thm:PWM} again, we have $g\in \pwm(\cLN_{p,\phi})$.
\end{proof}

\section*{Acknowledgement}
The authors would like to thank the referees for their careful reading 
and useful comments.
The first author was supported by Grant-in-Aid for Scientific Research (C), 
No.~24540159, Japan Society for the Promotion of
Science.
The second author was supported by Grant-in-Aid for Scientific Research (C), 
No.~24540171, Japan Society for the Promotion of
Science.


\bigskip

\begin{flushright}
\begin{minipage}{70mm}
\noindent
Eiichi Nakai \\
Department of Mathematics \\
Ibaraki University \\
Mito, Ibaraki 310-8512, Japan \\
enakai@mx.ibaraki.ac.jp  
\end{minipage}
\\[3ex]
\begin{minipage}{70mm}
\noindent
Gaku Sadasue \\
Department of Mathematics \\
Osaka Kyoiku University \\
Kashiwara, Osaka 582-8582, Japan \\
sadasue@cc.osaka-kyoiku.ac.jp  
\end{minipage}
\end{flushright}

\end{document}